\documentclass{amsart}
\usepackage{enumerate}
\usepackage{amsmath}
\usepackage{amsthm}
\usepackage{amssymb}
\usepackage{amsbsy}
\usepackage{amsfonts}
\usepackage{amstext}
\usepackage{amscd}

\usepackage{pgfplots}
\pgfplotsset{compat=1.15}

\usepackage{geometry}

\usepackage{mathrsfs}
\usetikzlibrary{arrows}

\usepackage{graphicx}
\numberwithin{equation}{section}
\theoremstyle{plain}
\newtheorem{thm}{Theorem}[section]
\newtheorem{prop}[thm]{Proposition}
\newtheorem{cor}[thm]{Corollary}

\theoremstyle{definition}
\newtheorem{exa}[thm]{Example}
\newtheorem{conj}[thm]{Conjecture}
\newtheorem{rem}[thm]{Remark}
\newtheorem{defi}[thm]{Definition}

\title{The graph energy game}
\author{Gerardo Arizmendi and Octavio Arizmendi}
\date{\today}

\begin{document}

\maketitle

\begin{center}\it{Dedicated to our father, Manuel Enrique Arizmendi San Pedro, \\on the occasion of his 65th birthday}
\end{center}

\begin{abstract}
    We study the graph energy from a cooperative game viewpoint. We introduce \emph{the graph energy game} and show various properties. In particular, we see that it is a superadditive game and that the energy of a vertex, as defined in Arizmendi and Juarez-Romero (2018), belongs to the core of the game. These properties imply new bounds for the energy of graphs. We also consider a version based on $p$-Schatten norms.
\end{abstract}


\section{Introduction}

In this work, we propose to study Gutman's graph energy \cite{Gut} from the viewpoint of cooperative game theory. For this, we define a cooperative game based on the graph energy and study its main properties and solutions. 

Cooperative game theory is a branch of game theory which models scenarios where different players may gain a larger utility by cooperating between them. One main problem is the decide how to divide the total utility so that every player is satisfied and willing to cooperate. There are various solutions to this problem depending on the desired properties one accepts as fair. The book \cite{Peters} presents an introduction to the topic.

On the other hand, recall that for a graph $G$, the energy of $G$, is given by  $$\mathcal{E}(G)=\sum^n_{i=1}|\lambda_i|$$
where $\lambda_1,\dots,\lambda_n$ denote the eigenvalues of the adjacency matrix of $G$, counted with multiplicity. The monograph  \cite{LSG}, presents a nice review of the theory. 

The idea is that for a given graph $G$ with vertex set $V$, we define a game on $V$ by considering the energy of the subgraphs of $G$. More precisely, for each subset $S$ of $G$, let $H=H(S)$ be the \emph{induced} subgraph of $G$. The value of $S$ is given by the energy of the subgraph $H$, $\mathcal{E}(H)$. We prove that this is a \emph{superadditive} game.

We would also be interested in a particular payoff: the energy of a vertex. The energy of vertex as defined in \cite{AJ} is a way to split the total energy of the graphs among the different vertices. In principle, it is natural to try to give an interpretation in terms of centrality, similar to the centrality measures studied in networks, such as  degree, PageRank, betweeness, closeness, etc.  However, from observation of examples, we believe that one is lead to the intuition that the vertex energy is closer to a payoff in a cooperative game. Thus, on the above described game we consider the payment given by the vertex energy and prove that it actually satisfies the axioms of symmetry, null player and efficiency. More important we show that this payment is in the core and thus provide evidence of the above intuition (see Section 2.3 for definitions). This is the main result of the paper.

We also consider generalizations where we consider instead of the energy, a new family of ``energies" based on the $p$-Schatten norms. Similar results are proved for these generalizations, obtaining new inequalities.

We hope that this paper motivates the study of graphs and their properties from the viewpoint of game theory.

Apart from this introduction, the paper contains five sections.  In Section 2 we present the necessary preliminaries. 2.1 and 2.2 are devoted to graphs and energy of graph and vertices, 2.3 for matrix inequalities and in 2.4 we describe basic concepts of cooperative game theory.  In Section 3, we prove the main results of the paper. We introduce a cooperative game based on energy  of a graph in 3.1. In 3.2 we consider the vertex energy from this viewpoint and give further considerations. In Section 4 we generalize these ideas  for $p$-Schatten norms by introducing the $p$-Schatten energy. We start with basic properties of this energy in 4.1, then introduce the $p$-energy of a vertex in 4.2, and finally  we specialize in the case $p=2$ in Section 4.3, proving that this case coincide with games studied before. In Section 5 we present some  examples. We finish with a final small section regarding conclusions and further questions.

\section{Preliminaries}

\subsection{Graphs and its adjacency matrix}

We will consider simple undirected finite graphs which we will call simply graphs.

A graph $G$ consists of a pair $(V,E):=(V(G),E(G))$ where $|V|<\infty $ and $E\subset V\times V$ is such that, $(v,v)\notin E$ for all $v\in V$, and $(v,w)\in E$ implies that  $(w,v) \in E$, for all $v,w\in V$.

A subgraph $H$ of $G$ is a graph such that $V(H)\subset V(G)$ and $E(H)\subset E(G)$. Given $S\subset{V(G)}$, the induced subgraph of $S$ is the subgraph $I(S)=(S,F)$ with vertex set $S$ and edge set $F=\{(v,w)\in E(G)|v\in S,w \in S\}.$ We call a subgraph $H$ of $G$, induced if there is some $S$ such that $H=I(S)$.

For any graph $G=(V,E)$, with vertex set $V=\{v_1,\dots,v_n\}$, the adjacency matrix of $G$, that we denote by $A=A(G)$, is the matrix with entries 
$A_{ij}=1$ if $(v_i,v_j)\in E$ and $A_{ij}=0$ otherwise.

\subsection{ Graph energy and vertex energy}

We denote by $M_n:=M_n(\mathbb{C})$, the set of matrices of size $n\times n$ over the complex numbers and by $Tr:M_n\to \mathbb{C}$, the trace on $M_n$, which is the map $A \mapsto A_{11}+\cdots+ A_{nn}$.


The energy of a graph $G$, denoted by $\mathcal{E}(G)$, is defined as the sum of the absolute values of the eigenvalues of the adjacency matrix $A = A(G)$, i.e.,
	\begin{equation*}
	\mathcal{E}(G) = \sum_{i=1}^{n}|\lambda_{i}|.
	\end{equation*}
In other words, the energy of a graph is the trace norm of its adjacency matrix. More precisely, for a matrix $M$, we define
the absolute value of $M$ by $|M|:=(MM^*)^{1/2}$ where $M^*$ denotes conjugate transpose of $M$. The energy of $G$ is then given by
\begin{equation*}
\mathcal{E}(G)= Tr(|A(G)|)=\displaystyle\sum_{i=1}^{n}|A(G)|_{ii}.
\end{equation*}

The next important definition is given in  \cite{AJ} and studied in \cite{AFJ}.
\begin{defi}\label{Energyv}

For a graph G and a vertex $v_i\in G$, \emph{the energy of the vertex} $v_i$, which we denote by $\mathcal{E}_G(v_i)$, is given by
\begin{equation}
  \mathcal{E}_G(v_i)=|A(G)|_{ii}, \quad\quad~~~\text{for } i=1,\dots,n,
\end{equation}
where $A(G)$ is the adjacency matrix of $G$.
\end{defi}

In this way the energy of a graph is given by the sum of the individual energies of the vertices of $G$,
\begin{equation*}
  \mathcal{E}(G)=\mathcal{E}_G(v_1)+\cdots+\mathcal{E}_G(v_n),
\end{equation*}
and thus the energy of a vertex is a refinement of the energy of a graph.

The idea behind the definition comes from non-commutative probability. Actually, the linear functional $\phi_i:M_n\to\mathbb{C}$ defined by $\phi_i(A)=A_{ii}$ is a vector state, and in particular, it is positive and $\phi(I)=1$, where $I$ is the identity matrix, see \cite{HoraObata} for details.

 H\"older inequality extends to the non-commutative case as follows: Let $(p,q)$ be a pair such that $1/p+1/q=1$, then
\begin{equation}
\label{Hold}\phi_i(XY)\leq \phi_i(|X|^p)^{1/p}\phi_i(|Y|^q)^{1/q}, \text{ for all } X,Y\in M_n.
\end{equation}

\subsection{Jensen's trace inequality}

Trace inequalities refer to inequalities for the trace of matrices or operators and functions of them.  In this paper we will use a specific trace inequality which is the generalization of Jensen's inequality for convex functions. Before introducing it, we need to bring some basic notation and definitions.

\begin{defi}
\begin{enumerate}
\item A matrix is called self-adjoint if $M=M^*$.
\item An orthogonal projection $P\in M_n$ is a self-adjoint matrix such that $P=P^2.$

\end{enumerate}
\end{defi}

For a self-adjoint matrix $M$, and a  real function $f$ we may define $f(M)$ as follows: consider the diagonalization $M=UDU^{-1}$ where $U$ is unitary and $D$ is diagonal, then 
$$f(M) := U \begin{bmatrix}
f(d_1) & \dots & 0 \\
\vdots & \ddots & \vdots \\
0 & \dots & f(d_n)
\end{bmatrix} U^{-1} $$ where $d_1, \dots, d_n$ denote the diagonal entries of $D$.

The main tool to prove the desired properties for the energy will be Jensen's trace inequality, which we state as a proposition for further reference.

\begin{prop}[Jensen's trace inequality] \label{JTI}

Let $f$ be a continuous convex function and let  $X_1,\dots, X_i\in M_n$ be selfadjoint. If  $f$ is convex, the following the inequality holds:
$$\operatorname{Tr}\Bigl(f\Bigl(\sum_{k=1}^iA_k^*X_kA_k\Bigr)\Bigr)\leq \operatorname{Tr}\Bigl(\sum_{k=1}^i A_k^*f(X_k)A_k\Bigr),$$
where $A_1,\dots, A_i\in M_n$ is any collection of matrices such that $\sum_{k=1}^iA_k^*A_k=I.$
\end{prop}

\subsection{Cooperative Game Theory}

In this section, we consider a cooperative game with transferable utility, or a TU- game. That is  a pair $(N,w)$, where $N=\{1,\dots,n\}$ with $n\in\mathbb N$ and $w:2^N\rightarrow\mathbb R$ such that $w(\emptyset)=0$. $N$ is called the set of players and the function $w$ is called the characteristic function. Each subset of $N$ is called a coalition. For each coalition $S\subset N$, $w(S)$ is called the value of the $S$.
\\
{\bf Remark:}
Usually  the characteristic function is denoted by $v$, in this article we use $w$ so that the exposition is clear, since we work also with graphs in which $v_i$ denotes the $i$-th vertex of a graph.

\begin{defi}
A game $(N,w)$ is called {\it superadditive} if $w(S\cup T)\geq w(S)+w(T)$ whenever $S\cap T=\emptyset$.
\end{defi}

\begin{defi}
A game $(N,w)$ is called {\it convex} if $w(S\cup T)+w(S\cap T) \geq w(S)+w(T)$ for every $S,T\subset N$.
\end{defi}

A {\it payoff distribution} or {\it payoff vector }is an element $x=(x_1,\dots,x_n)$ in $\mathbb{R}^n$ and $x_i$ is called the payoff of the player $i$.

 For a payoff distribution $x$ and a coalition $S$, we use the notation $x(S)=\sum_{i\in S} x_i$.

\begin{defi} If $\mathbf{x}$ is a payoff vector we will say that

\begin{enumerate}
    \item $x$ is {\it individually rational} if $x_i\geq w({i})$ 
    \item  $x$ is {\it group rational} if $x(S)\geq w(S)$
    \item  $x$ is {\it efficient}  if $x(N)=w(N).$ 
    \item  $x$  is an {\it imputation} if it is individually rational and efficient. The set of imputations of $(N,w)$ is denoted by $I(w)$. 
\end{enumerate}

\end{defi}

\begin{defi}
 The \it{core} of a game $(N,w)$, denoted by $C(w)$ is the set of imputations that are group rational. That is, 
 $$C(w)=\{x| x(S)\geq w(S) \text{ for each } S\subset N \text{ and } x(N)=w(N) \}$$

\end{defi}

An important concept in cooperative game theory is the Shapley value. Given the set of all games of $n$ players, which we denote by $\mathcal G^{\mathcal N}$, a \emph{value} is a map from $\phi:\mathcal G^{\mathcal N}\mapsto \mathbb{R}^n$, i.e., a \emph{value} is a function that assigns a payoff distribution to each game of $n$ players.

Some desirable properties that values may satisfy are the following:

\begin{enumerate}
    \item {\bf Efficiency}: $\sum_{i=1}^n\phi_i(w)=w(N)$ for all $w\in \mathcal G^{\mathcal N}$,  where $\phi_i(w)$ denotes the the $i$-th entry of $\phi(w)$.
    \item {\bf Symmetry}: $\phi_i(w)=\phi_j(w)$ for all $w\in \mathcal{G}^{\mathcal N}$ and all symmetric players $i$ and $j$
in $w$, where $i$ and $j$ are symmetric in $(N,w)$ if $w(S\cup\{i\})=w(S\cup\{j\})$ for every coalition $S\subset N\backslash\{i,j\}$.
    \item {\bf Additivity}: For two games $w_1,w_2\in G^{\mathcal N}$,  $\phi(w_1+w_2)=\phi(w_1)+\phi(w_2).$ 
    \item {\bf Null player}: $\phi_u(w)=0$ for all $w\in \mathcal{G}^{\mathcal N}$ and all null players  $i\in w$,
    where a player is called null if $w(S\cup{i})-w(S)=0$ for every coalition $S$.
\end{enumerate}

The Shapley value is the only value that satisfies the axioms of efficiency, symmetry, additivity and null player, it is well known that actually the Shapley value is given by

\[\phi_i(w)=\sum_{S\subset N\backslash\{i\}}\frac{|S|!(n-|S|-1)!(w(S\cup\{i\})-w(S))}{n!}.\]

In general, it is not direct to decide whether the Shapley value is in the core. One quite used criteria to decide this is convexity: for any convex game, the Shapley value is always in the core. We refer the reader to \cite{Peters} for details.

\section{The energy game on graphs}\label{EGG}

In this section we define the energy game on graphs. We address natural questions about this game. Namely, we prove that for each graph, the energy game is a superadditive game, although it is not always convex. We also prove some important properties about the solutions, namely, we show that the core is non empty. Moreover, the energy of a vertex as defined in Definition \ref{Energyv}, is a payoff distribution, which we prove is in the core. These are the main results of the article. 

The energy of a vertex, thought as a value restricted the energy game on graphs, also satisfies the axioms of null player, symmetry and efficiency. However, it is easy to see that the Shapley value does not always coincide with the energy of a vertex.

\subsection{The energy game}

 We now define the object of our interest, namely, the energy game on graphs. 

\begin{defi}
Let $G=(V,E)$ be a simple undirected graph. For each subset $S\subset V$ let $w(S):=\mathcal{E}(I(S))$ where $I(S)$ is the graph induced by $S$. We denote $(V,w)$, the \emph{energy game on $G$.}  
\end{defi}

The following theorem justifies the interest in the definition above.
\begin{prop} \label{Super} For any graph $G=(V,E)$, the game $(V,w)$ is superadditive.
\end{prop}
\begin{proof}
The proposition is equivalent to proving that if the adjacency matrix of a graph $G$ is partitioned as 
$$A(G)=\left(\begin{array}{cc}
X&Y\\
Z&W
\end{array}\right),$$
with $X$ and $W$ square matrices of dimension $k$ and $n-k$, respectively, then  $Tr(|X|)+Tr(|W|)\leq Tr(|A(G)|)$. This may be proved by the use of Lemma 4.17 of \cite{LSG} (see also Theorem 3 in \cite{Tho}). We will prove it using Jensen's trace inequality, Proposition \ref{JTI}, applied to the function $f=Abs$. Indeed, taking   $n=2$, $X_1=A(G)$, $X_2=A(G)$,  $A_1=P$ and $A_2=I-P$, where $P$ is the projection to the first $k$ entries, we get that
\begin{equation}\label{jen2} \operatorname{Tr}\Bigl(|PA(G)P +(I-P)A(G)(I-P)| \Bigr) \leq \operatorname{Tr}\Bigl(P| A(G)|P+(1-P)| A(G)|(1-P)\Bigr).\end{equation}
Since we have that $$PA(G)P=\left(\begin{array}{cc}
X&0\\
0&0
\end{array}\right), (I-P)A(G)(I-P)=\left(\begin{array}{cc}
0&0\\
0&W
\end{array}\right),$$ 
the left hand side of \eqref{jen2} equals $Tr(|X|)+Tr(|W|).$ On the other hand for any matrix $M$, $M$ and $PMP+(I-P)M(I-P)$ have the same diagonal entries, thus $Tr(PMP+(I-P)M(I-P))=Tr(M)$, and thus the right handside of \eqref{jen2} is $Tr(|A(G)|)$.
\end{proof}

\begin{rem}
At first glance, one may think that superadditivity may be related to the triangle inequality of the  1-Schatten norm, see section 4. Triangle inequality, also known as Ky-Fan's theorem,
says that for any matrices $A+B$
$$Tr(|A+B|)\leq Tr(|A|)+Tr(|B|).$$
This may seem in contradiction with the inequality $E(S\cup T)\geq E(S)+E(T)$, but looking closer at the proof of Proposition \ref{Super} one sees that triangle inequality yields $E(S\cup T)\leq E(S)+E(T) +E(U)$
where $U$ is the bipartite subgraph of $G$ which contains the edges in $G$ between the induced subgraphs by the sets $S$ and $T$.  
\end{rem}

Let us mention that the energy game is not always convex as the following example shows.

\begin{exa}
Consider $G=(V,E)$, with $V=\{1,2,3\}$ and $E=\{\{1,2\},\{2,3\}\}$. Let $S=\{1,2\}$ and $T=\{2,3\}$. Then $w(S)=w(T)=2$, $w(S\cup T)=2\sqrt{2}$ and  $w(S\cap T)=0$, hence 
$w(S)+w(T)>w(S\cup T)+w(S\cap T)$. This proves that this game is not convex.
\end{exa}

\subsection{The vertex energy as a payoff}

In \cite{AJ}, the energy of a vertex is introduced and considered. For a graph $G$, with adjacency matrix $A(G)$, and a vertex $v_i$, recall that the energy of the vertex $v_i$ is given by 
$$\mathcal{E}_G(v_i)=|A(G)|_{ii}.$$

The vector $e=(\mathcal{E}_G(v_1),\dots, \mathcal{E}_G(v_n))$ is an imputation since the sum of  the energy of the vertices is exactly the total energy of the graph. 

Now we prove the main result of the paper, which states that the vertex energy is a payoff which is in the core. 

This will be a direct consequence of the following theorem.

\begin{thm}\label{MTeorem} Let $H$ be an induced subgraph of $G$, then 
$$\sum_{v\in V(H)} \mathcal E_G(v)\geq\mathcal E(H),$$
 where $\mathcal E_G(v)$ is the vertex energy of $v$ in $G$ and $\mathcal E(H)$ is the energy of $H$. 
 \end{thm}

\begin{proof}
Without loss of generality, suppose the the vertex set of $H$ is $\{1,\dots,k\}$ for some $k$.
Let $A$ be the adjacency matrix of $G$, let $P$ be the projection into the first $k$ entries, and $Q=I-P$.  Now, Jensen's trace inequality, Proposition \ref{JTI}, applied to the function $f=Abs$, the absolute value, and the parameters   $n=2$, $X_1=A$, $X_2=0$,  $A_1=P$ and $A_2=Q$, gives 
\begin{equation}\label{mainequation}
Tr(|PAP|)\leq Tr(P|A|P).    
\end{equation}
Since $PAP$ is the adjacency matrix of the induced subgraph $H$ the left-hand side of \ref{mainequation} coincides with the energy of $H$, $E(H)$.  On the other hand, by definition the diagonal of the matrix $|A|$ consists of the values $\mathcal E(v_i)$, $i=1,\dots,|V|$ and thus the diagonal of  $P|A|P$ contains the values $(P|A|P)_{ii}=\mathcal E(v_i)$, for $i=1,\dots,k$ and  $(P|A|P)_{ii}=0$ for $i=k+1,\dots,n$, which means that $Tr(P|A|P)=\sum_{v\in V(H)} \mathcal E_G(v)$

\end{proof}
We are now in position to prove the main theorem.
 
\begin{thm}\label{Core}
Let $(G,V)$ be a simple undirected graph, then the payoff vector given by
$e=(\mathcal{E}(v_1),\dots, \mathcal{E}(v_n))$ is in the core of the game
$(V,w_{\mathcal{E}})$.
\end{thm}
\begin{proof} We need to prove that for each $S\subset V$,  $w_{\mathcal E}(S)\geq e(S)$ that is 
$$\mathcal E(I(S))\leq \sum_{i\in S} \mathcal E(v_i)$$ where $\mathcal E(v_i)$ is the vertex energy of $v_i$ in $G$ and $\mathcal E(G(S))$ is the energy of of the graph induced by $S$. The result now follows from Theorem \ref{MTeorem}.
\end{proof}

{\bf Remark:} Theorem \ref{MTeorem} gives an inequality  useful for bounding the energy of a graph. For example, for two connected vertices $v$ and $w$ the theorem implies that $\mathcal E(v)+\mathcal E(w)\geq 2$, this has been proved independently in \cite{Arizmendi-Arizmendi}. Another example is given by Theorem 4.20 in \cite{LSG} for which we give an alternative proof.

\begin{thm}
If $F$ is an edge cut of a simple graph, then $\mathcal{E}(G-F)\leq \mathcal{E}(G)$.
\end{thm}
\begin{proof}
Since  $F$ is an edge cut of $G$, $G-F=H\cup K$, where $H$ and $K$ are two complementary induced graphs. We can split the set vertices of $G$ as $V_H\cup V_K$ where the $V_H$ is the set of vertices of $H$ and $V_K$ is the set of vertices in $H$.
By Theorem \ref{MTeorem}, $$\mathcal{E}(H)\leq \sum_{v\in V_H} \mathcal E(v), \text{ and} $$
$$\mathcal{E}(K)\leq \sum_{v\in V_K}, \mathcal E(v).$$
Hence,
$$\mathcal{E}(G-F)=\mathcal{E}(H)+\mathcal{E}(K)\leq \sum_{v\in V_H} \mathcal E(v)+ \sum_{v\in V_K} \mathcal E(v)=\sum_{v\in V(G)} \mathcal E(v)=\mathcal{E}(G)$$

\end{proof}

If one considers the energy game on graphs with $n$ vertices, then the energy of a vertex can be thought as a value in these games. In this sense, one can ask if this value coincides with the Shapley value restricted to the energy game on graphs. The next theorem goes in this direction
\begin{thm}\label{syefnull}
The energy of a vertex satisfies the axioms of symmetry, efficiency and null player axiom.
\end{thm}

\begin{proof}
The efficiency axiom follows by definition.

We will prove that the only null players are isolated vertices. Indeed, suppose that $v_i$ is not isolated, then $deg(i)>0$ and by \cite[Theorem 3.3]{AFJ}    
$$\mathcal{E}(v_i)\geq deg(v_i)/\Delta>0,$$
where $\Delta$ denotes the maximum degree of the graph. 
On the other hand, by  Theorem \ref{MTeorem}, for $H=G\setminus\{i\}$ we have 
$$\sum_{v\in V\setminus\{v_i\}} \mathcal E_G(v)\geq\mathcal E(G\setminus\{v_i\})=w(V\setminus\{v_i\}),$$
and then $$w(V)=\mathcal E_G(v_i)+\sum_{v\in V\setminus\{v_i\}}\mathcal{E}(v_i) >w(V\setminus\{v_i\}),$$ proving that $i$ is not a null player.
Now if $v_i$ is isolated $\mathcal{E}(v_i)=0$, as desired.

Now, let us characterize symmetric players. For a vertex $v$ in $V$ consider $N(v)$, the set of neighbours of $v$.
We claim that $v_i$ and $v_j$ are symmetric if and only if $N(v_i)=N(v_j)$. First if $N(v_i)=N(v_j)$, then for any $S\subset G$ the sub-graph by $S\cup{v_i}$ and $S\setminus{v_i}$ are isomorphic to  $S\cup{v_j}$ and $S\setminus{v_j}$, respectively. Thus $v_i$ and $v_j$ are symmetric.

Now, let $v_i$ and $v_j$ be symmetric players and consider for each $w\in N(v_i)$ the subgraph induced by $\{w,v_i\}$, which has energy $2$. Since $v_j$ and $v_i$ are symmetric then the subgraph induced by $\{w,v_j\}$ must also have energy $2$. This, of course, means that $w$ is connected to $v_j$ in $G$. Since this works for all $w\in N(v_j)$, then $N(v_j)\subset N(v_i)$. Interchanging the roles of $v_i$ and $v_j$, we see that $N(v_i)\subset N(v_j)$ and thus $N(v_i)=N(v_j)$.
\end{proof}

Even though the energy of a vertex satisfies the above axioms of symmetry, efficiency and null player, it does not coincide with the Shapley value, as the following example shows. 

\begin{exa}
Consider $G=(V,E)$, with $V=\{1,2,3\}$ and $E=\{\{1,2\},\{2,3\}\}$, i.e. $G$ is the path with $3$ vertices. A direct calculation shows that $\mathcal{E}(v_1)=\mathcal{E}(v_3)=\frac{1}{\sqrt{2}}$ and
$\mathcal{E}(v_2)=\sqrt{2}$. On the other hand, the values of each coalition are given by
$$\mathcal{E}(\{\})=\mathcal{E}(\{1\})=\mathcal{E}(\{2\})=\mathcal{E}(\{3\})=\mathcal{E}(\{1,3\})=0,$$

$$\mathcal{E}(\{1,2\})=\mathcal{E}(\{2,3\})=2,$$

$$\mathcal{E}(\{1,2,3\})=2\sqrt{2}.$$

Hence, if we denote by
$\varphi(v_i)$ the Shapley value of the $i$-th vertex then a direct calculation shows that
$\varphi(v_1)=\varphi(v_3)=\frac{2\sqrt{2}-1}{3}$ and
$\varphi(v_2)=\frac{2+2\sqrt{2}}{3}$.

\end{exa}

However, we conjecture  that the Shapley value is in the core. 
\begin{conj}
For the energy game, the Shapley value is in the core.
\end{conj}

Let us end with the following lemma which says that the marginal contribution of a vertex in a coalition is larger than the sum of the  vertex energies of the coalition, which is a step in that direction.  
\begin{prop}
Let $S\subset V$ and $i\in S$. The marginal contribution of $i$, satisfies $$ w(S)-w(S\setminus \{i\})\geq \mathcal{E}_S(i).$$ 
\end{prop}
\begin{proof}
This follows from the following inequalities
\begin{eqnarray} 
\mathcal{E}(S)&=&\sum_{j\in S}\mathcal{E}_S(j)=\sum_{j\in S\setminus\{i\} }\mathcal{E}_S(j)+\mathcal{E}_S(i)
\\&\geq& 
\mathcal{E}(S\setminus \{i\})+\mathcal{E}_S(i),
\end{eqnarray} 
where we used Theorem 3.5 in the inequality. 
\end{proof}

\section{Schatten energy game}

Recalling that the energy of a graph is the 1-Schatten norm (or nuclear norm) of the adjacency matrix, in this section we want to consider an energy inspired by the $p-$Schatten norms. 

Explicitly, instead of using $|A|$ one can use $|A|^p$. In this way, the $p$-energy of a graph is defined as $\mathcal E_p(G):= tr(|A(G)|^p)$ where $A(G)$ denotes the adjacency matrix of $G$. By the fact that $A(G)^2_{ij}$ counts the walks of size two from the vertex $v_i$ to the vertex $v_j$, one easily gets that $\mathcal E_2(G)= 2m$, where $m$ is the number of edges in the graph. 

Here, we should point out that using the $p$-Schatten norm, $(tr(|A(G)|^p))^{1/p}$,  directly would result in a theory which is not superadditive.

The $p$-energy game on a graph $(G,v)$ is defined in an analogous way. For each subset $S\subset V$ let $w_p(S):=\mathcal{E}_p(G(S))$ where $G(S)$ is the graph induced by $S$. Moreover for $p\geq1$ we have that

\begin{prop}
Let $p\geq 1$. The $p$-energy game on graphs is superadditive.
\end{prop}

\begin{proof}
The proof is the same as the one of Proposition \ref{Super} by using Jensen's trace inequality for $n=2$, $X_1=A(G)$, $X_2=A(G)$,  $A_1=P$ and $A_2=I-P$, where $P$ is the projection to the first $k$ entries. In this case, one should consider the convex function $f(\cdot)=|\cdot|^p$. 
\end{proof}

\subsection{Some properties of the Schatten energy}

Apart from the above properties described in relation to game theory, by being the $p$-th power of the Schatten norm we automatically have the following properties.

\begin{prop}[Monotonicity of Schatten norms] \label{mon schatten}
For $1<p<q<\infty$, we have that 
\begin{equation} \label{monotonocity}
 \mathcal{E}_p^{1/p}\geq \mathcal{E}_q^{1/q}. \end{equation}
\end{prop}


A direct corollary is the following, which will be improved below for bipartite graphs.

\begin{cor}
 Let $G=(V,E)$ be a graph with $|V|=n$ and $|E|=m$ then 
 
 i) If $1\leq p\leq 2$, then $(2m)^{p/2}\leq \mathcal{E}_p(G)$.
 
 ii)  If $p>2$, then $(2m)^{p/2} \geq\mathcal{E}_p(G)$.
\end{cor}

\begin{proof}
Since $\mathcal E_2(G)=2m$, by \eqref{monotonocity} we have,
$$\mathcal{E}_p(G)^{1/p} \geq\mathcal{E}_2(G)^{1/2}=(2m)^{1/2}$$
for $0\leq p<2$, and
 $$\mathcal{E}_p(G)^{1/p} \leq\mathcal{E}_2(G)^{1/2}=(2m)^{1/2}$$
 for $p<2$.
\end{proof}

On the other hand, multiplying by a factors of $n^{-1/p}$ and $n^{-1/q}$  reverses the  inequality (\ref{monotonocity}).

\begin{prop} \label{monotonocity2}
For $1<p<q<\infty$, we have that 
\begin{equation}\label{Hol} (\mathcal{E}_p/n)^{1/p}\leq (\mathcal{E}_q/n)^{1/q}. \end{equation}
\end{prop}

In particular taking again $q=2$ we obtain generalization of McClelland inequality \cite{Mc}, and taking $p=2$ we obtain a lower bound in terms of the number of edges $m$

\begin{prop}
 Let $G=(V,E)$ be a graph with $|V|=n$ and $|E|=m$ then 
 
 i) If $1\leq p\leq 2$, then $ \mathcal{E}_p(G)\leq n^{1-p/2} (2m)^{p/2}$.
 
 ii)  If $p>2$, then  $ \mathcal{E}_p(G)\geq n^{1-p/2} (2m)^{p/2}$.
\end{prop}

\begin{proof}
Note that $\mathcal{E}_2(G)=2m$. Both inequalities follow then from (\ref{Hol}).
\end{proof}

It is well known that among trees with $n$-vertices, $S_n$ and $P_n$ are the graphs with  minimal and maximal energy, respectively. 

For the Schatten energy there is a different behavior between $1\leq p\leq2$ and $p>2$. This is stated in Li, Shi and Gutman \cite{LSG} as a private communication from S. Wagner and is asked again in Nikiforov\cite{Nik3}.

\begin{conj} [\cite{Nik3}]
 Let $T=(V,E)$ be a tree with $|V|=n$ and $|E|=m$ then 
 
 i) If $1\leq p\leq 2$, then $\mathcal{E}_p (S_n)\leq \mathcal{E}_p(T) \leq\mathcal{E}_p(P_n)$.
 
 ii)  If $p>2$, then $\mathcal{E}_p (P_n) \leq \mathcal{E}_p(T) \leq\mathcal{E}_p(S_n)$.
\end{conj}


We are able to prove the inequalities for the star $S_n$, including all bipartite graphs.
\begin{prop}
 Let $G=(V,E)$ be a bipartite graph with $|V|=n$ and $|E|=m$ then 
 
 i) If $1\leq p\leq 2$, then $2(m)^{p/2}\leq \mathcal{E}_p(G)$.
 
 ii)  If $p>2$, then $2(m)^{p/2} \geq\mathcal{E}_p(G)$.
 
 In particular if for any tree $T$ $$\mathcal{E}_p(S_n) \geq\mathcal{E}_p(T)$$ if $p<2$ and $$\mathcal{E}_p(S_n) \geq\mathcal{E}_p(T).$$
\end{prop}

\begin{proof}

Since $G$ is bipartite, the eigenvalues $\lambda_1\geq \lambda_2\cdots \geq\lambda_n$
satisfy that $\lambda_i=-\lambda_{n-i}$.

Now, $\mathcal{E}_2(T)=2m=\lambda_1^2+\cdots+\lambda_n^2$, then
$$\lambda_1^2+\cdots+\lambda_k^2=m=n-1$$
and 
$$\mathcal{E}(G)_p=
2(\lambda_1^p+\cdots+\lambda_k^p).$$

Now, by the monotonicity of $p$-norms we see that 
$$ \sqrt{m}=(\lambda_1^2+\cdots+\lambda_k^2)^{1/2}\leq (\lambda_1^p+\cdots+\lambda_k^p)^{1/p}$$
if  $1\leq p\leq2$, and 
$$ \sqrt{m}=(\lambda_1^2+\cdots+\lambda_k^2)^{1/2}\geq (\lambda_1^p+\cdots+\lambda_k^p)^{1/p}$$
if  $p\geq2.$

Finally, if $G$ is a tree then $m=n-1$. Recalling that the star graph has non zero eigenvalues $-\sqrt{n-1}$ and $\sqrt{n-1}$, we see that its $p$-energy is given by $2\sqrt{n-1}^p=2(m)^{p/2}$. 
\end{proof}

\begin{rem}
It is not true in general that $2(m)^{p/2}< \mathcal{E}_p(G)$, for $p>2$. For $n\geq3$, the complete graph on $n$-points is an easy counterexample. In this case the eigenvalues are $-1$ with multiplicity $n-1$ and $n-1$ with multiplicity 1. Thus the Schatten energy is given by $\mathcal{E}_p(K_n)=(n-1)^p+(n-1)$. On the other hand, $m=n(n-1)/2$. It can be easily checked, by elementary calculus that $(n-1)^p+(n-1)-(n(n-1)/2)^{p/2}>0$ for $p>4,n\geq3$, and thus $\mathcal{E}_p(K_n)>(2m)^{p/2}$. 
\end{rem}

\subsection{The $p$-energy of a vertex} 
Let us fix a graph $G$ for the sake of notation. One can also define the $p$-energy of a vertex as $\mathcal E_p(v_i):=|A(G)|^p_{ii}$. A direct calculation easily proves that $\mathcal{E}_2(v_i)=d_i$. One can prove that it is in the core for $p\geq1$ 

\begin{thm}\label{Th1}
Let $p\geq1$. The $p$-energy of a vertex is in the core.
\end{thm}

\begin{proof}
The proof is analogous as the proof of Theorem \ref{Core}, where $p=1$ is considered. The only modification to be done is to consider the convex function  $f=|\cdot|^p$, instead of $f=|\cdot|$.
\end{proof}

\begin{prop}The $p$-energy  of a vertex is an imputation that satisfies the axioms of null player, efficiency and symmetry.
\end{prop}

\begin{proof}

It is clear from the definition that it is an imputation. The other properties are proved along the same lines as Theorem \ref{syefnull}, with obvious modifications. We leave the proof to the reader. 
\end{proof}
We do not go into much details on further of the graph theoretical properties of the $p$-energy of a vertex but let us state a basic bound.

\begin{thm}
Let $0<r<s$. Then $$\mathcal E_r(v_i)\leq\mathcal{E}_s(v_i)^{r/s}.$$
In particular, for $p<2$,
$$\mathcal E_p(v_i)\leq d_i^{p/2},$$
where $d_i$ denotes the degree of $v_i.$
\end{thm}
\begin{proof}
We use Non Commutative Holder Inequality (\ref{Hold}) for $X=|A|^r$ and  $Y=1$, with $p=s/r$ and $q=s/(s-r)$.  For $p=2$, we notice that $\mathcal E_2(v_i)=d_i.$
\end{proof}

Finally, we show that as in the case of  vertex energy, the $p$-energy of a bipartite graph is divided evenly between its two parts. 
\begin{prop}\label{bipartite}
Let $G$ be a bipartite graph with parts $V$ and $W$ then
\begin{equation}
\sum_{v \in V} \mathcal{E}_{p} (v) = \sum_{w \in W} \mathcal{E}_{p} (w).
\end{equation}
\end{prop}

\begin{proof}
Being bipartite corresponds to having a matrix of the form,
\begin{equation}
A =
\begin{pmatrix}
{\bf 0}_{r,r} & B \\
 B^T    & {\bf 0}_{s,s}
\end{pmatrix},
\end{equation}
then the  matrix $M=AA^T$ is of the form
\begin{equation}
M =
\begin{pmatrix}
BB^T & 0\\
 0   & B^T B
\end{pmatrix},
\end{equation}
and the p-th positive power of the absolute value of the matrix is given by
\begin{equation}
|A|^P =
\begin{pmatrix}
(\sqrt{BB^T})^p & 0\\
 0   &(\sqrt{B^T B}P
\end{pmatrix}
\end{equation}
Now, we have the equalities
$$Tr(B^TB)^{p/2}=Tr\sqrt{B^TB}^p= \sum_{i\in|V_2|} |A|^p_{ii} =\sum_{i\in|V_2|} E(v_i),$$
and
$$ Tr(B^TB)^{p/2}=Tr\sqrt{B^TB}^p=\sum_{i\in|V_1|} |A|^p_{ii} =\sum_{i\in|V_1|} E(v_i),$$
from where, since $BB^T$ and $B^TB$ have the same eigenvalues, then $Tr(BB^T)^p/2=Tr(B^TB)^p/2$, and we get the result.
\end{proof}

\subsection{The case $p=2$}

For $p=2$ the $p-$energy is the same as the two times the number of edges of the graph. In this case the $p-$energy game is a particular case of what is called \emph{ induced subgraph games} with all weights identically 1, see \cite{ref1,ref2} for details.

It is well known that induced subgraph games with non-negative weights the game is convex, as a corollary the core is nonempty.

Moreover, in this case the $p-$energy of the vertex coincides with the degree which by theorem \ref{Th1} is in the core. Finally we have the next proposition.

\begin{prop}
The Shapley value of the $2$-energy game on graphs is exactly the degree of the vertex, which also coincides with the $2$-energy of a vertex.
\end{prop}

\begin{proof}
Let $\mathcal V(i)$ the set of neighbours of the vertex $i$.  Recall that the Shapley value is given by
$$\phi_i(v)=\sum_{S\subset N\backslash\{i\}}\frac{|S|!(n-|S|-1)!(w(S\cup\{i\})-w(S))}{n!}$$

Since $w(S\cup\{i\})$ is twice  the number edges of the induced graph of $S\cup\{i\}$ and $w(S)$ is twice the number of edges of the induced graph of $S$, then the marginal contribution of i, $w(S\cup\{i\})-w(S)$, equals two times the number of neighbours of $i$ inside $S$ i.e. $2|S\cap \mathcal V(i)|$, which yields,
\begin{eqnarray*}
\phi_i(v)&=&\sum_{S\subset N\backslash\{i\}}\frac{|S|!(n-|S|-1)!2|S\cap \mathcal V(i)|}{n!}\\
&=&2\sum_{S\subset N\backslash\{i\}}\frac{|S|!(n-|S|-1)!|S\cap \mathcal V(i)|}{n!}\\
&=&2\sum_{S\subset N\backslash\{i\}}\left(\sum_{j\in S\cap \mathcal V(i)}\frac{|S|!(n-|S|-1)!}{n!}\right)
\end{eqnarray*}

Now since the only coalitions that contribute to the sum are the ones containing a neighbour we can actually rewrite this sum as

\begin{eqnarray*}
\phi_i(v)&=&2\sum_{j\in \mathcal{V}}\left(\sum_{S\subset N\backslash\{i\},j\in S}\frac{|S|!(n-|S|-1)!}{n!}\right)
\end{eqnarray*}

Finally, the number of coalitions of size $k$ containing the vertex $j$ and not containing the vertex $i$ is given by $\binom{n-2}{k-1}$ so that

\begin{eqnarray*}
\phi_i(v)&=&2\sum_{j\in \mathcal{V}}\left(\sum_{k=1}^{n-1}\binom{n-2}{k-1}\frac{k!(n-k-1)!}{n!}\right)\\
&=&2\sum_{j\in \mathcal{V}}\left(\sum_{k=1}^{n-1}\binom{n-2}{k-1}\frac{k!(n-k-1)!}{n!}\right)\\
&=&2\sum_{j\in \mathcal{V}}\left(\sum_{k=1}^{n-1}\frac{k}{n(n-1)}\right)\\&=&2\sum_{j\in \mathcal{V}}\frac{1}{2}=deg(i).
\end{eqnarray*}

\end{proof}

\section{Examples}

Let us finish with two simple examples of well known graphs.

\begin{exa}[Path]

Consider a path of size $n$. That is the graph with vertex set $V=\{1,2,\dots, n\}$ and edge set $E=\{i\sim i+1|i=1\dots,n\}$.

\begin{center}
\begin{figure}[ht]
\includegraphics[width=10cm]{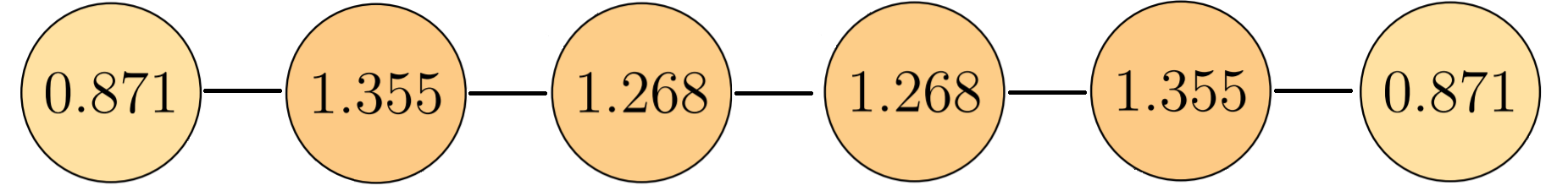}
\caption{Vertex energies on the path of size 6. The intensity of the color is proportional to this quantity.}
\end{figure}
\end{center}

According to Corollary 4.10 in \cite{ALR} the vertex energies satisfy the following relation. 
$$\mathcal{E}(v_1)<\mathcal{E}(v_3)< \cdots <\mathcal{E}(v_{2k+1})< \mathcal{E}(v_{2k+2})< \mathcal{E}(v_{2k})< \cdots <\mathcal{E}(v_{4}) < \mathcal{E}(v_2)$$
for all $k<[n/4].$

So the minimum is attained at the vertices $v_1$ and $v_n$, and the maximum at $v_2$ and $v_{n-1}$.

A possible game-theoretical interpretation for this behavior is the following. Since the extreme points ($v_1$ and $v_n$) can only cooperate with their neighbour, they have not much negotiation margin and thus accept a payment which is much lower than the rest. So even though the total gain of players $v_1$ and $v_2$ is larger bigger than 2, player $v_1$ accepts a payoff which is smaller than 1.
\end{exa}

\begin{exa}

The star with $n$ vertices is an example of a bipartite graph with parts $\{v_1\}$ and $\{v_2,\dots,v_{n-1}\}$.
\begin{figure}[h]
\includegraphics[width=12cm]{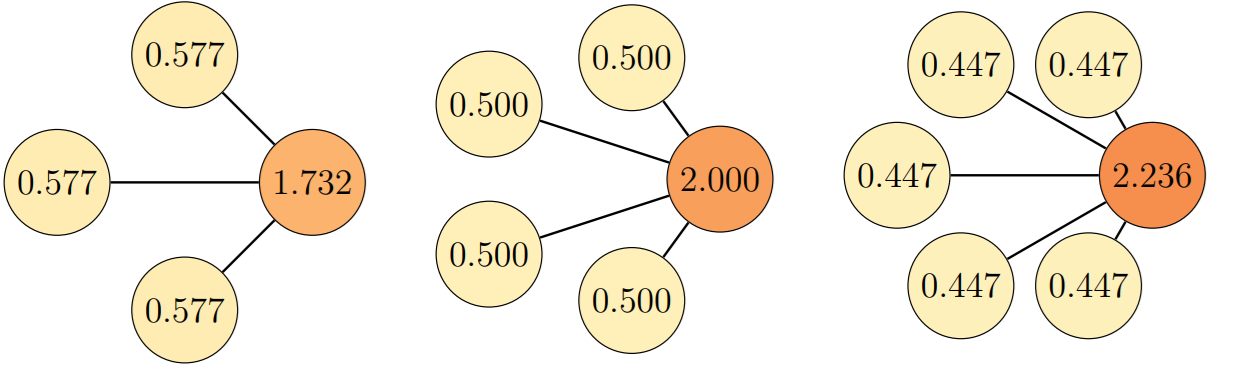}
\caption{Vertex energies on star of size 4 to 6. The intensity of the color is proportional to this quantity.}
\end{figure}

The total energy is $2\sqrt{n-1}$ and is divided on equally between the two parts. So $\mathcal{E}_G(v_1)=\sqrt{n-1}$ and $\mathcal{E}_G(v_i)=1/\sqrt{n-1}$. Moreover, one can think here that each of the edges represents a collaboration of work and in this case it is split evenly between the two vertices.

\end{exa}

\begin{exa} 

Consider the graph $S_3=P_3$, with vertex set $v_1,v_2,v_3$ and with edges $v_1\sim v_2$ and $v_2\sim v_3$ . The core is the set of payoffs $\phi:V\to\mathbb{R}$ that satisfies the equations 
\begin{eqnarray*}
\phi(v_1)&\geq& 0, \quad \phi(v_2)\geq 0, \quad \phi(v_3)\geq 0, \\
\phi(v_1)+\phi(v_2)&\geq& 2, \quad \phi(v_2)+\phi(v_3)\geq 2, \quad \phi(v_2)+\phi(v_3)\geq 0, \\
\phi(v_1)+\phi(v_2)+\phi(v_3)&=& 2\sqrt{2}\\
\end{eqnarray*}

The vertex energy is given by $\mathcal{E}(v_1)=1/\sqrt{2}$, $\mathcal{E}(v_2)=\sqrt{2}$, $\mathcal{E}(v_3)=1/\sqrt{2}$, while the Shapley value is given by $Sh(v_1)=Sh(v_3)=.6093$ ad $Sh(v_2)=1.6093$ which is inside the core. Figure \ref{coreS3} illustrates the different values.

\begin{center}

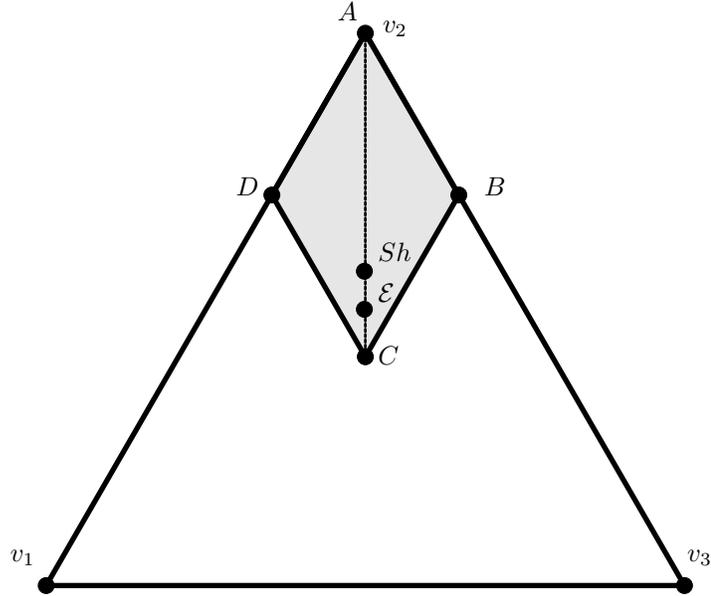
\begin{figure}[h]

\definecolor{uuuuuu}{rgb}{0,0,0}

\begin{tikzpicture}[line cap=round,line join=round,>=triangle 45,x=1cm,y=1cm,scale=3]

\fill[line width=2pt,fill=black,fill opacity=0.1] (1.4142135623730954,2.4494897427831783) -- (1,1.7320508075688772) -- (1.4142135623730954,1.014611872354576) -- (1.8284271247461905,1.7320508075688774) -- cycle;
\draw [line width=2pt] (0,0)-- (1.4142135623730954,2.4494897427831783);
\draw [line width=2pt] (1.4142135623730954,2.4494897427831783)-- (2.8284271247461907,0);
\draw [line width=2pt] (2.8284271247461907,0)-- (0,0);
\draw [line width=1pt] (1.8284271247461905,1.7320508075688774)-- (1.4142135623730954,1.014611872354576);
\draw [line width=1pt] (1,1.7320508075688772)-- (1.4142135623730954,1.014611872354576);
\draw [line width=1pt,dash pattern=on 1pt off 1pt] (1.4142135623730954,2.4494897427831783)-- (1.4142135623730954,1.014611872354576);
\draw (2.8,0.2) node[anchor=north west] {$v_3$};
\draw (1.45,2.5448385846905963) node[anchor=north west] {$v_2$};
\draw (1.43,1.56) node[anchor=north west] {$Sh$};

\draw (1.43,1.38) node[anchor=north west] {$\mathcal{E}$};
\draw (1.43,1.10) node[anchor=north west] {$C$};
\draw (1.90,1.85) node[anchor=north west] {$B$};
\draw (.80,1.85) node[anchor=north west] {$D$};

\draw (1.25,2.63) node[anchor=north west] {$A$};
\draw (-0.2,0.2) node[anchor=north west] {$v_1$};
\draw [line width=2pt] (1.4142135623730954,2.4494897427831783)-- (1,1.7320508075688772);
\draw [line width=2pt] (1,1.7320508075688772)-- (1.4142135623730954,1.014611872354576);
\draw [line width=2pt] (1.4142135623730954,1.014611872354576)-- (1.8284271247461905,1.7320508075688774);
\draw [line width=2pt] (1.8284271247461905,1.7320508075688774)-- (1.4142135623730954,2.4494897427831783);
\begin{scriptsize}
\draw [fill=uuuuuu] (0,0) circle (1pt);
\draw [fill=uuuuuu] (1,1.7320508075688772) circle (1pt);
\draw [fill=uuuuuu] (1.4142135623730954,2.4494897427831783) circle (1pt);
\draw [fill=uuuuuu] (1.4142135623730954,1.014611872354576) circle (1pt);
\draw [fill=uuuuuu] (1.8284271247461905,1.7320508075688774) circle (1pt);
\draw [fill=uuuuuu] (2.8284271247461907,0) circle (1pt);
\draw [fill=uuuuuu] (1.41,1.3939356435643566) circle (1pt);
\draw [fill=uuuuuu] (1.41,1.225) circle (1pt);
\end{scriptsize}
\end{tikzpicture}

    \caption{The core in gray, delimited by the cuadrilateral $ABCD$ with  $A=(0,2.828,0)$, $B=(0.828,2,0)$, $C=(0.828,1.172,0.828)$ $D=(0,2,0.828)$. The Shapley value $Sh=(0.6093,1.6093,0.6093)$  and the energy of a vertex $\mathcal{E}= (0.707,1.414,0.707)$ }
    \label{coreS3}
\end{figure}

\end{center}
\end{exa}

\section{Conclusions and further questions} 

In this paper we have shown how to define game by  borrowing the concept of energy from chemical graph theory and generalized for Schatten norms. The main contribution of this paper is giving a successful interpretation of the vertex energy of a graph as a payoff and thus giving an explanation on how the energy is split among the vertices, which otherwise may seems counter-intuitive. As a consequence of proving desirable properties from the viewpoint of game theory, we obtained new inequalities which may be of interest from the graph theoretical viewpoint.  This opens a new line of study, which we expect to develop in the future.

Regarding the Schatten energy, the case $p=2$ is special since in that case the the $p$-energy and  the Shapley value coincide, and thus the solution may be extended to any game. In general, the $p$-energy makes sense for all matrices but it is not clear how to define it for any game, it will useful to give another set of axioms which characterize the $p$-energy and which is suitable for any game. Moreover, as we mentioned above, the Shapley value and the $p$-energy of a vertex share many properties as solutions of the energy game. Giving inequalities between them which may allow one to compare them may be of interest.

Finally, as the reader may have noticed, the main technical tool used in this paper is Jensen Trace Inequality. This inequality is, more generally, valid for any convex function, giving sense for other games on graphs, which satisfy superadditivity and have at least one solution on the core analogue to the $p$-energy. One natural question is which set of games come as an outcome of this construction and how large is this set as a subset of the games on $n$ players.

\subsection*{Acknowledgement} The first author would like to thank Carlo Danon for the work done during the summer of 2019 which was crucial for starting this paper. We would like to thank Jorge Fern\'andez for some comments and for letting us use the Figures 1 and 2 in this paper. The second author received support from CONACYT grant CB-2017-2018-A1-S-9764.

{\small
\renewcommand{\baselinestretch}{0.5}
\newcommand{\bi}{\vspace{-.05in}\bibitem}

\end{document}